\documentclass[10pt,a4paper]{amsart} %titlepage,openany,openbib,book,scrbook,amsbook,fleqn
\usepackage{amssymb,amsthm,amsmath,amsfonts,amscd}
\usepackage{graphicx}
\usepackage{url}
\allowdisplaybreaks[1]
\usepackage[english]{babel}
\usepackage{enumerate}
\usepackage{cite}
\usepackage{bbm}
\usepackage{stmaryrd}
\usepackage{mathrsfs}
\newcommand{\norm}[1]{\left\lVert#1\right\rVert}

\newcommand{\abs}[1]{\lvert#1\rvert}
\newcommand{\lrabs}[1]{\left\lvert#1\right\rvert}

\newcommand{\lrbr}[2]{\left\langle#1,{#2}\right\rangle}

\newcommand{\Hscr} {{\mathscr H}}

\newcommand{\up}[1]{\textup{#1}}

\renewcommand{\tilde}{\widetilde}

\renewcommand{\d}{\textup{d}}
\newcommand{\D}{\textup{D}}

\newcommand{\eps}{{\varepsilon}}
\renewcommand{\phi}{\varphi}

\newcommand{\R}{\mathbbm{R}}
\newcommand{\N}{\mathbbm{N}}

\renewcommand{\div}{\operatorname{div}}

\newcommand{\supp}{\operatorname{supp}}

\newcommand{\vol}{\operatorname{vol}}
\newcommand{\dist}{\operatorname{dist}}

\newcommand{\dom}{\operatorname{dom}}

\newcommand{\rs}[1]{{\restriction}_{#1}}

\newcommand{\loc}{\textup{loc}}

\newcommand{\sgn}{\operatorname{sgn}}

\newcommand{\BIGOP}[1]{\mathop{\mathchoice%
{\raise-0.22em\hbox{\huge $#1$}}%
{\raise-0.05em\hbox{\Large $#1$}}{\hbox{\large $#1$}}{#1}}}

\newcommand{\BIGboxplus}{\mathop{\mathchoice%
{\raise-0.35em\hbox{\huge $\boxplus$}}%
{\raise-0.15em\hbox{\Large $\boxplus$}}{\hbox{\large $\boxplus$}}{\boxplus}}}

\numberwithin{equation}{section}%section

\newtheorem{theorem}{Theorem}[section] %section

\theoremstyle{plain}

\newtheorem*{bem*}{Bemerkung}

\newtheorem{lemma}[theorem]{Lemma}
\newtheorem{definition}[theorem]{Definition}
\newtheorem{remark}[theorem]{Remark}

\newcommand{\ol}{\overline}

\newcommand{\Sgn}{\operatorname{Sgn}}

\title[Convergence of the $p$-Laplace evolution equation]{Convergence of solutions to the $p$-Laplace evolution equation as $p$ goes to $1$}
\author[J. M. T\"olle]{Jonas M. T\"olle}
\address{Department of Mathematics and Systems Analysis, Aalto University, PO Box 11100, 00076 Aalto, Finland}
\email{jonas.tolle@aalto.fi}
\date{\today}

\begin{document}
\begin{abstract}
We prove that the set of solutions to the parabolic singular $p$-Laplace equation with Dirichlet boundary conditions on a bounded Lipschitz domain $\Omega$ for all space dimensions is continuous in the parameter $p\in [1,+\infty)$ and the initial
data. The highly singular limit case $p=1$ is included. In particular, we show that
the solutions $u_p$ converge strongly in $L^2(\Omega)$, uniformly in time, to the solution $u_1$ of the parabolic $1$-Laplace equation as $p\to 1$.
\end{abstract}

\maketitle

\section{Introduction}

Let $\Omega\subset\R^d$, $d\in\N$ be an open bounded domain with Lipschitz boundary $\partial\Omega$.
Consider the following family of singular diffusion equations indexed by $p\in(1,+\infty)$,
\begin{equation}\tag{$p_D$}\left\{\begin{aligned}
\dfrac{\partial u_p}{\partial t}-\div\left[\abs{\nabla u_p}^{p-2}\nabla u_p\right]&= f_p&&\text{in}\quad(0,+\infty)\times\Omega,\\
u_p&=0&&\text{on}\quad(0,+\infty)\times\partial\Omega,\\
u_p(0)&=x_p,&&\end{aligned}\right.
\end{equation}
where $x_p\in L^2(\Omega)$ and $f_p\in L^2_\loc((0,+\infty);L^2(\Omega))$.
The notation ($p_D$) refers to Dirichlet boundary conditions.

Consider also
\begin{equation}\tag{$1_D$}\left\{\begin{aligned}
\dfrac{\partial u_1}{\partial t}-\div\left[\Sgn\left(\nabla u_1\right)\right]&\ni f_1&&\text{in}\quad(0,+\infty)\times\Omega,\\
u_1&=0&&\text{on}\quad(0,+\infty)\times\partial\Omega,\\
u_1(0)&=x_1,&&\end{aligned}\right.
\end{equation}
where $x_1\in L^2(\Omega)$ and $f_1\in L^2_\loc((0,+\infty);L^2(\Omega))$ and
the vector-valued (and multi-valued) sign-operator $\Sgn:\R^d\to 2^{\R^d}$
is defined by
\[\Sgn(x):=\left\{\begin{aligned}&\dfrac{x}{\abs{x}},&&\;\;\text{if}\;\;x\in\R^d\setminus\{0\},\\
                  &\left\{y\in\R^d\;\vert\;\abs{y}\le 1\right\},&&\;\;\text{if}\;\;x=0.
                 \end{aligned}\right.\]
In fact, the expression ``$\div\left[\Sgn\left(\nabla \cdot\right)\right]$'' is
merely heuristic, we shall later replace it by the subdifferential $\partial\Phi_1$,
where the convex potential $\Phi_1$ is defined in \eqref{phi1defi} below.

In this work, we consider solutions to equation ($p_D$), $p\in [1,+\infty)$,
in the sense of Br\'ezis, that is, strong solutions of the associated
subgradient evolution equation in $L^2(\Omega)$, see Definition \ref{soldefi} below. We are interested
in continuity properties of ($p_D$) in the initial data and, in particular,
when the parameter $p$ varies. The \emph{$p$-Laplace evolution} ($p_D$) belongs to
the class of gradient flow-type equations, where the dynamics is generated
by the (infinite dimensional) gradient of a differentiable functional.

In order to prove
the desired continuity of solutions in $p$, we shall apply a theorem
by Attouch, stating that so-called \emph{Mosco continuity} of convex,
lower semi-continuous potentials implies the continuous dependence
of the solutions to the related, subgradient-type evolution equations, see \cite{A3}.
The classical theorem requires a Hilbert space framework, which explains
our choice of $L^2$-data. For Banach spaces, variational convergence and viscosity solutions
are brought together in \cite{A4}.

Inclusion ($1_D$) requires some attention. It is called \emph{total variation flow}, see \cite{ABCM,ACM}. We shall
model it by the gradient of the total-variation functional
on functions of bounded variation in $L^2(\Omega)$.

We shall also consider the $p$-Laplace evolution with Neumann boundary conditions,
\begin{equation}\tag{$p_N$}\left\{\begin{aligned}
\dfrac{\partial u_p}{\partial t}-\div\left[\abs{\nabla u_p}^{p-2}\nabla u_p\right]&= f_p&&\text{in}\quad(0,+\infty)\times\Omega,\\
\dfrac{\partial u_p}{\partial\nu}&=0&&\quad\text{on}\quad(0,+\infty)\times\partial\Omega,\\
u_p(0)&=x_p,&&\end{aligned}\right.
\end{equation}
where $\nu$ is the generalized outward normal on $\partial\Omega$.

Also
\begin{equation}\tag{$1_N$}\left\{\begin{aligned}
\dfrac{\partial u_1}{\partial t}-\div\left[\Sgn\left(\nabla u_1(t)\right)\right]&\ni f_1(t)&&\text{in}\quad(0,+\infty)\times\Omega,\\
\dfrac{\partial u_1}{\partial\nu}&=0&&\quad\text{on}\quad(0,+\infty)\times\partial\Omega,\\
u_1(0)&=x_1,&&\end{aligned}\right.
\end{equation}
where again the expression ``$\div\left[\Sgn\left(\nabla \cdot\right)\right]$''
is made rigorous as a subdifferential $\partial\Psi_1$ of the potential
$\Psi_1$ defined in \eqref{psi1defi}.

The singular diffusion operator in equations ($p_D$), ($p_N$) is called
\emph{$p$-Laplacian}. Both equations of evolution-type are covered by the theory of nonlinear semigroups associated to equations of subdifferential-type
in Hilbert space as in the book of Br\'ezis \cite{Brez}. For the notion of solutions used in
this paper, see Definition \ref{soldefi} below. Explicit representations for the subdifferential operators were
given by Schuricht \cite{Schu} for inclusion $(1_D)$ and by Andreu, Ballester, Caselles and Maz\'on \cite{ABCM} for inclusion $(1_N)$.

Classically, the $p$-Laplace equation appears in geometry,
quasi-regular mappings, fluid dynamics and plasma physics, see \cite{DiBe,Diaz}.
In \cite{Lad}, Lady\v{z}enskaja suggests the $p$-Laplace evolution as a model of motion
of non-Newtonian fluids. A typical $2$-dimensional application can be found in
image restoration, see \cite[Ch. 3]{AuKo}
for a comprehensive treatment.
General equations of $p$-type are studied in
\cite{Vaz2}.

Eigenvalue problems for the $1$-Laplacian have been studied e.g. by Fridman, Kawohl, Schuricht and Parini \cite{Fri,Schu,KaSchu,Par}.

The stochastic analog to ($p_D$) has been examined by Liu \cite{Liu1}, respectively, to ($1_D$) by Barbu, Da Prato and R\"ockner \cite{BDPR}
and by Barbu and R\"ockner \cite{BaRoe}.
In the case of stochastic equations, similar results to this paper have been proved by Ciotir and the author in \cite{CioToe1,CioToeCorr}.
Singular stochastic diffusion equations, including the total variation flow, with gradient Stratonovich noise have been
studied in \cite{CioToe2,Toe3}, and the stability of solutions to singular stochastic equations has been investigated
by Gess and the author in \cite{GessToe}. The proof in this paper, however, originates
from the author's thesis \cite[Ch. 8.3]{Toe2}, where a variable-space approach, similar to
that in \cite{AttCom}, is used.
Previous convergence results for the stationary problem have been obtained by Mercaldo, Rossi, Segura de Le\'on and Trombetti in
\cite{MSdLT,MSdLT2,MRSdLT}. Convergence of the evolution problem for local solutions has been previously investigated
by Giga, Kashima and Yamazaki in \cite{GKY}.

Since the first version of this work appeared as a preprint in 2011, solutions to evolutionary problems with linear growth, including the total variation flow,
have been obtained as limits of $p$-growth problems as $p\downarrow 1$ for Dirichlet \cite{Boegelein2019,GianazzaKlaus} and Neumann \cite{Schaetzler2019} boundary conditions,
see also \cite{Boegelein2013,Boegelein2014,Boegelein2016,Kinnunen2022}, and \cite{BeckSchmidt} for the stationary Dirichlet problem for linear growth functionals in $BV$, while the inhomogeneous, nonlocal and $\mathbb{A}$-variation versions of the total variation flow
have been studied in \cite{SeguradeLeon2015,Latorre2022,Latorre2026,Li2024,Li2025,Mazon2023,Meyer2026}, and stochastic versions in
\cite{BaRoe,Gess2014,GessRoe,GessToe,Gess2016b,CioToe2,Toe3,Seib2023}.

Let us formulate the main result of this paper. Its proof can be found in Section~\ref{msec} below. The notion of solution is specified in Section~\ref{msecs} below,
see Definition \ref{soldefi}.
\begin{theorem}\label{mainthm}
Let $p_0\in[1,+\infty)$, $\{p_n\}\subset (1,+\infty)$ with $\lim_n p_n=p_0$. Let $T>0$.
Suppose that $\lim_n\norm{x_{p_n}-x_{p_0}}_{L^2(\Omega)}=0$ and $\lim_n\int_0^T\norm{f_{p_n}(t)-f_{p_0}(t)}^2_{L^2(\Omega)}\,\d t=0$.
Denote by $u_p$ the solution to equation $(p_D)$ in the sense of Definition \ref{soldefi}
and by $v_p$ the solution to equation $(p_N)$ in the sense of Definition \ref{soldefi}.

Then
\begin{equation}
\lim_n\sup_{t\in [0,T]}\norm{u_{p_n}(t)-u_{p_0}(t)}_{L^2(\Omega)}=0
\end{equation}
and
\begin{equation}
\lim_n\sup_{t\in [0,T]}\norm{v_{p_n}(t)-v_{p_0}(t)}_{L^2(\Omega)}=0.
\end{equation}
\end{theorem}

In other words, if
\[F:[1,+\infty)\times L^2(\Omega)\times L^2_\loc((0,+\infty);L^2(\Omega))\to C([0,+\infty);L^2(\Omega))\]
denotes the map
that assigns the solution $t\mapsto u_p^{x_p,f_p}(t)$ to equation ($p_D$), our result states that
$F$ is continuous. Certainly, a similar statement holds for $(p_N)$.

\section{The related energies}

We observe that, for $p\in (1,+\infty)$, the $p$-Laplace operator on $L^2(\Omega)$ is
a gradient-type operator with the convex potential
\[\Phi_p(u):=\begin{cases}\dfrac{1}{p}\displaystyle\int_\Omega\abs{\nabla u}^p\,\d x,\;&\text{if}\;u\in W^{1,p}_0(\Omega)\cap L^2(\Omega),\\
+\infty,\;&\text{if}\;u\in L^2(\Omega)\setminus W^{1,p}_0(\Omega).\end{cases}\]
Note that the Dirichlet boundary conditions are encoded into the energy space
$W^{1,p}_0(\Omega)$ which is the standard first order $p$-Sobolev space on $\Omega$ such
that its elements have vanishing trace (see next paragraph).

In general, for an element $u\in W^{1,p}(\Omega)$, we define
its \emph{trace} $\gamma_p(u)$ to $\partial\Omega$ by
\begin{equation}\label{Wtracedefi}\begin{split}
&\int_{\partial\Omega}\gamma_p(u)\lrbr{\eta}{\nu}\,\d\Hscr^{d-1}\\
&=\int_\Omega u\div\eta\,\d x
+\int_\Omega\lrbr{\eta}{\nabla u}\,\d x\qquad\forall\eta\in C^1(\ol{\Omega};\R^d),\end{split}
\end{equation}
where $\nu$ is the generalized outward normal on $\partial\Omega$ and $\Hscr^{d-1}$ is the $d-1$-dimensional Hausdorff measure on $\partial\Omega$.
We have that
$\gamma_p:W^{1,p}(\Omega)\to L^p(\partial\Omega,\Hscr^{d-1})$ is a continuous linear operator.
Note that our definition of $W^{1,p}_0(\Omega)$ coincides with the standard one,
where one takes $W^{1,p}_0(\Omega)$ to be the closure of $C_0^\infty(\Omega)$
in $W^{1,p}(\Omega)$, see \cite[\S 5.5, Theorem 2]{Evan} for $C^1$-boundaries; the proof carries over to Lipschitz boundaries, cf. \cite{Leo}. We denote by
$\partial_i$, $\nabla$ resp. the weak or ordinary partial derivative in direction
of the $i$-th coordinate, the weak or ordinary gradient resp. $\D_i$ and $\D$
refer to the corresponding objects in the sense of Schwartz distributions.
Sometimes we shall write $[Df]$ instead of $Df$ to indicate that $[Df]$ is a
Radon measure.
Of course,
if $f\in L^1_\loc(\Omega)$ and $\D f\in L^1_\loc(\Omega;\R^d)$ then
$\D f=\nabla f$ a.e.

Define also the energy of the $p$-Laplace operator in $L^2(\Omega)$ with Neumann boundary conditions by
\[\Psi_p(u):=\begin{cases}\dfrac{1}{p}\displaystyle\int_\Omega\abs{\nabla u}^p\,\d x,\;&\text{if}\;u\in W^{1,p}(\Omega)\cap L^2(\Omega),\\
+\infty,\;&\text{if}\;u\in L^2(\Omega)\setminus W^{1,p}(\Omega).\end{cases}\]

Let us continue investigating $\Phi_p$.
It is easily seen that $\Phi_p$ is a proper
convex function. It is well-known that $\Phi_p$ is lower semi-continuous, see
e.g. \cite{Butt}.

The corresponding subdifferential $\partial \Phi_p$ is a realization of the $p$-Laplace operator on $\Omega$ with Dirichlet boundary conditions.
Recall that, in general, for a convex, proper, lower semi-continuous function $\Psi:H\to (-\infty,+\infty]$
on a separable Hilbert space $H$, the subdifferential $\partial\Psi$, defined by
$[x,y]\in\partial\Psi$ iff
\[(y,z-x)_H\le\Psi(z)-\Psi(x)\quad\forall z\in H,\]
is a maximal monotone graph.

For smooth functions $\phi\in C_0^\infty(\Omega)$ such that the right-hand side below belongs to $L^2(\Omega)$
(which is automatic if $p\ge 2$), we have that $\phi\in\dom(\partial\Phi_p)$ and
\[(\partial \Phi_p)(\phi)=-\div\left[\abs{\nabla\phi}^{p-1}\sgn(\nabla\phi)\right],\]
where $\sgn$ is the selection of $\Sgn$ with minimal Euclidean norm.
Note that for $1<p<2$ not every $\phi\in C_0^\infty(\Omega)$ belongs to $\dom(\partial\Phi_p)$, due to
the singularity of the $p$-Laplacian at critical points of $\phi$.

Note that the extension of $\Phi_p$ to $W^{1,p}_0(\Omega)$ is Fr\'echet differentiable in $W^{1,p}_0(\Omega)$-norm, a fact which we, however,
will not use. By the theory of subdifferential operators, $\dom(\partial\Phi_p)$ is
dense in $L^2(\Omega)$, see e.g. \cite[Ch. 2, Proposition 2.6]{Barb2}.

We shall continue investigating the limit case $p=1$.
We need some facts from the space $BV$, which we shall collect here.

\begin{definition}
A function $f\in L^1_\loc(\Omega)$ is said to be of \textup{bounded variation} if
\[\norm{\D f}(\Omega):=\sup\left\{\int_\Omega f\div\eta\,\d x\;\bigg\vert\;\eta\in C_0^\infty(\Omega;\R^d),\;\norm{\eta}_\infty\le 1\right\}<+\infty.\]
The value $\norm{\D f}(\Omega)$ is called the \textup{total variation} of $f$.
The space of all classes of functions in $L^1(\Omega)$ that are of bounded variation is denoted by $BV(\Omega)$.
\end{definition}

Let $f\in BV(\Omega)$. Then there is a Radon measure $\mu_f$ on $\Omega$ and a measurable function $\sigma_f:\Omega\to\R^d$
such that $\abs{\sigma_f}=1$ $\mu_f$-a.e. and
\begin{equation}\label{radoneq}\int_\Omega f\div\eta\,\d x=-\int_\Omega\lrbr{\eta}{\sigma_f}\,\d\mu_f\quad\forall\eta\in C_0^1(\Omega;\R^d).\end{equation}
By \cite[\S 5.1]{EG}, $\norm{\D f}(\Omega)=\mu_f(\Omega)$.
Set $\d[\D f]:=\sigma_f\mu_f(\d x)$, which is a $\R^d$-valued Radon measure on $\Omega$.
Hence \eqref{radoneq} becomes
\begin{equation}\label{radon2eq}
\int_\Omega f\div\eta\,\d x=-\int_\Omega\lrbr{\eta}{\d[\D f]}\quad\forall\eta\in C_0^1(\Omega;\R^d).
\end{equation}
For each $f\in BV(\Omega)$ there is the \textit{trace} $\gamma_1(f)$ to $\partial\Omega$
satisfying the following formula
\begin{equation}\label{tracedefi}\int_\Omega f\div\eta\,\d x=-\int_\Omega\lrbr{\eta}{\d[\D f]}+\int_{\partial\Omega}\gamma_1(f)\lrbr{\eta}{\nu}\,\d\Hscr^{d-1}
\quad\forall\eta\in C^1(\ol{\Omega};\R^d),
\end{equation}
where $\nu$ is the outward normal and $\Hscr^{d-1}$ is the Hausdorff measure on $\partial\Omega$.
We have that $\gamma_1(f)\in L^1(\partial\Omega,\Hscr^{d-1})$.
We refer to \cite[Ch. 5.3]{EG} and \cite[Ch. 3]{AFP} for details.

Also, we have that $W^{1,p}(\Omega)\subset BV(\Omega)$, $p\in [1,+\infty]$ and
\begin{equation}\label{HBV}
\int_\Omega\abs{\nabla u}\,\d x=\norm{\D u}(\Omega)\quad\forall u\in W^{1,p}(\Omega),
\end{equation}
see \cite[Ch. 5.1, Example 1]{EG}.

\begin{remark}\label{Rdrem}
Suppose that $\partial\Omega$ is Lipschitz. Let $u\in BV(\Omega)$ and extend $u$ by zero outside $\Omega$. Then $u\in BV(\R^d)$ and
\[\norm{\D u}(\R^d)=\norm{\D u}(\Omega)+\int_{\partial\Omega}\abs{\gamma_1(u)}\,\d\Hscr^{d-1}\]
where $\gamma_1(u)\in L^1(\partial\Omega,\Hscr^{d-1})$ is the trace of $u$. We refer to \up{\cite[Theorem 3.87]{AFP}}.
\end{remark}

We are ready to define the convex potential associated to the $1$-Laplacian in $L^2(\Omega)$ with
Dirichlet boundary conditions. The Dirichlet problem was derived by Kawohl and Schuricht in \cite{Schu,KaSchu}.

Define
\begin{equation}\label{phi1defi}\Phi_1(u):=\begin{cases}
           \norm{\D u}(\R^d),&\;\;\text{if}\;\; u\in BV(\Omega)\cap L^2(\Omega),\\
+\infty,&\;\;\text{if}\;\; u\in L^2(\Omega)\setminus BV(\Omega),
          \end{cases}\end{equation}
where $u$ is extended by zero outside $\Omega$, so that, by Remark \ref{Rdrem},
$\Phi_1(u)=\norm{\D u}(\Omega)+\int_{\partial\Omega}\abs{\gamma_1(u)}\,\d\Hscr^{d-1}$ for $u\in BV(\Omega)\cap L^2(\Omega)$.
$\Phi_1$ is obviously proper and convex. Lower semi-continuity follows
by standard arguments via integration by parts since $\norm{\D u}(\R^d)$ is lower semi-continuous
with respect to weak convergence in $L^1_\loc(\R^d)$.

The $1$-Laplacian is defined to be the subdifferential $\partial\Phi_1$ of $\Phi_1$.
As above, $\dom(\partial\Phi_1)$ is dense in $L^2(\Omega)$. A precise representation of $\partial\Phi_1$ is given in
\cite[Theorem 2.1]{Schu} and \cite[Theorem 4.23]{KaSchu}.

Define also
\begin{equation}\label{psi1defi}\Psi_1(u):=\begin{cases}
           \norm{\D u}(\Omega),&\;\;\text{if}\;\; u\in BV(\Omega)\cap L^2(\Omega),\\
+\infty,&\;\;\text{if}\;\; u\in L^2(\Omega)\setminus BV(\Omega).
          \end{cases}\end{equation}
$\Psi_1$ is convex, proper and lower semi-continuous, see \cite{ACM}.

\section{Mosco convergence}\label{msecs}

Let $H$ be a separable Hilbert space.
Let $\Phi_n:H\to(-\infty,+\infty]$, $n\in\N$, $\Phi:H\to(-\infty,+\infty]$
be proper, convex, l.s.c. functionals.

Recall the following definition.
\begin{definition}[Mosco convergence]\label{moscodefi}
We say that $\Phi_n\xrightarrow[]{M}\Phi$ in the \emph{Mosco sense} if
\[\tag{M1}
\forall x\in H\;\forall x_n\in H,\; n\in\N,\; x_n\rightharpoonup x\;\text{weakly in $H$}:
           \quad\liminf_n \Phi_n(x_n)\ge\Phi(x).\]
\[\tag{M2}
\forall y\in H\;\exists y_n\in H,\; n\in\N,\; y_n\to y\;\text{strongly in $H$}:
           \quad\limsup_n \Phi_n(y_n)\le\Phi(y).\]
\end{definition}

Consider the following sequence of abstract gradient-type evolution equations in $H$:
\begin{equation}\left\{\begin{aligned}
\frac{\d}{\d t}u_n(t)+\partial\Phi_n(u_n(t))&\ni f_n(t),&&0<t<+\infty,\\
u_n(0)&=x_n,&&\end{aligned}\right.
\end{equation}
where $x_n\in H$ and $f_n\in L^2_\loc((0,+\infty);H)$.

Also,
\begin{equation}\label{sols}\left\{\begin{aligned}
\frac{\d}{\d t}u(t)+\partial\Phi(u(t))&\ni f(t),&&0<t<+\infty,\\
u(0)&=x,&&\end{aligned}\right.
\end{equation}
where $x\in H$ and $f\in L^2_\loc((0,+\infty);H)$.

A solution to, say, equation \eqref{sols} is defined as follows.
Compare with \cite[D\'efinition 3.1]{Brez}.
\begin{definition}\label{soldefi}
We call $u\in C([0,+\infty);H)$ a \emph{solution} to equation \eqref{sols} if
$u$ is locally absolutely continuous on $(0,+\infty)$,
$u(0)=x$,
$u(t)\in\dom(\partial\Phi)$ for
a.e. $t>0$ and \eqref{sols} holds for a.e. $t>0$ or, equivalently, in
the sense of distributions in $L^2_\loc((0,+\infty);H)$.
\end{definition}
By \cite[Th\'eor\`eme 3.6]{Brez}, for every $x\in\ol{\dom(\Phi)}$ and every $f\in L^2_\loc((0,+\infty);H)$ there is a unique solution to \eqref{sols}
in the sense of Definition \ref{soldefi}. If $x\notin\dom(\Phi)$, absolute continuity in general only holds on $[\delta,T]$ for $0<\delta<T$,
which is why we require local absolute continuity on the open interval $(0,+\infty)$ only.

We shall need the following theorem.
\begin{theorem}\label{Moscothm}
Suppose that for $T>0$ we have that $x_n\in\ol{\dom(\Phi_n)}$, $x\in\ol{\dom(\Phi)}$, $x_n\to x$ strongly in $H$, and $f_n\to f$ strongly
in $L^2((0,T);H)$.

Then
$\Phi_n\xrightarrow[]{M}\Phi$ implies that
$u_n\to u$ strongly in $H$, uniformly on $[0,T]$.
\end{theorem}
\begin{proof}
By \cite[Theorem 3.66]{A}, Mosco convergence $\Phi_n\xrightarrow[]{M}\Phi$ implies the convergence of $\partial\Phi_n$ to $\partial\Phi$ in the sense of graphs.
Since strong convergence in $L^2((0,T);H)$ implies strong convergence in $L^1((0,T);H)$, the assertion follows from \cite[Theorem 3.74]{A}.
\end{proof}
In our applications, $H=L^2(\Omega)$ and all domains $\dom(\Phi_p)$, $\dom(\Psi_p)$, $p\in[1,+\infty)$, are dense in $L^2(\Omega)$,
since they contain $C_0^\infty(\Omega)$. Hence the condition on the initial data is always satisfied.
See \cite{A2,A3,AttBeer,CT} for related results.

\section{Proof of Theorem \ref{mainthm}}\label{msec}

Let $p_0\in [1,+\infty)$, $p_n\in (1,+\infty)$, $n\in\N$ such that $\lim_n p_n=p_0$.
Write $q_n:=p_n/(p_n-1)$, $n\in\N$, $q_0:=p_0/(p_0-1)$ (with $1/0:=+\infty$). If we
can prove that
\begin{equation}\label{toprove}
\Phi_n:=\Phi_{p_n}\xrightarrow[n\to +\infty]{M}\Phi_{p_0}=:\Phi,\end{equation}
the assertion of
Theorem \ref{mainthm} follows from Theorem \ref{Moscothm}.

Before we prove \eqref{toprove}, we shall need the following approximation result.
\begin{lemma}[Littig--Schuricht]\label{lilem}
Let $\Omega\subset\R^d$ be open and bounded with Lipschitz boundary $\partial\Omega$ and let $u\in BV(\Omega)\cap L^p(\Omega)$ for some $p\in [1,+\infty)$. Then there is a sequence
$\{u_k\}$ in $C_0^\infty(\Omega)$ such that, for any $\tilde{p}\in [1,p]$,
\begin{equation}
\norm{u_k-u}_{L^{\tilde{p}}(\Omega)}\to 0\quad\text{and}\quad\norm{\D u_k}(\R^d)\to\norm{\D u}(\R^d).
\end{equation}
\end{lemma}
The result was proved by
Littig and Schuricht in \cite[Theorem 3.2]{LiSchu}, see also \cite[Satz 2.31]{Li}.
Previously, if $\partial\Omega$ is $C^2$, the result was proved by
Parini, see \cite[Theorem 2.11]{Par} and also \cite{Par2}.

\begin{proof}[Proof of \eqref{toprove}]
We start with proving (M1) from Definition \ref{moscodefi}. Let $u_n,u\in L^2(\Omega)$, $n\in\N$, such that $u_n\rightharpoonup u$ weakly in $L^2(\Omega)$.
In particular, by the uniform boundedness principle,
\begin{equation}\label{pnsuppn}\sup_n\norm{u_n}_{L^{2}(\Omega)}<+\infty.\end{equation}
We may assume that
\[\liminf_n \Phi_{n}(u_n)<+\infty,\]
since otherwise there is nothing to prove.
Extract
a subsequence (also denoted by $\{u_n\}$) such that
\[\lim_n\Phi_n(u_n)=\liminf_n \Phi_n(u_n)\quad\text{and}\quad C:=\sup_n\Phi_n(u_n)<+\infty.\]
Since $u_n\in W^{1,p_n}_0(\Omega)$,
we can extend $u_n$ to $\R^d$ by zero outside $\Omega$ (denoted also by $u_n$) and get that
$u_n\in W^{1,p_n}(\R^d)$ with $\nabla u_n=0$ a.e. outside $\Omega$ (cf. \cite[Exercise 15.26]{Leo}).
Similarly, we extend $u$ by zero outside $\Omega$.

Fix an open ball $B\subset\R^d$ such that $\ol{\Omega}\subset B$. By the Cauchy--Schwarz inequality,
\[\int_B\abs{u_n}\,\d x\le(\vol B)^{1/2}\sup_n\norm{u_n}_{L^2(\Omega)},\]
and by H\"older's inequality,
\begin{equation}\label{Dieq}\begin{split}\int_B\abs{\nabla u_n}\,\d x&\le\left(p_n \Phi_n(u_n)\right)^{1/p_n}(\vol B)^{1/q_n}\\
&\le(1+C\sup_n p_n)(1+\vol B).
\end{split}\end{equation}
Hence $\{u_n\}$ is bounded in $W^{1,1}(B)$.
By compactness of the embedding $W^{1,1}(B)\subset L^1(B)$,
a subsequence of $\{u_n\}$ converges in $L^1(B)$ to some $f\in L^1(B)$ (see e.g. \cite[\S 1.4.6, Lemma]{Maz} or \cite[Theorem 6.2]{Ada}).
Since $u_n=0$ outside $\Omega$, we can extend $f$ by zero outside $B$ and get that $u_n\to f$ in $L^1(\R^d)$.
Also, the measures
$\{\partial_i u_n\,\d x\}$ are vaguely bounded and hence vaguely relatively compact, see \cite[Paragraphs 46.1, 46.2]{B}.
For each $1\le i\le d$,
we can extract a subsequence, such that $\{\partial_i u_n\,\d x\}$ converges vaguely to some locally finite Radon
measure $m_i$ on $\R^d$. By vague convergence
and integration by parts,
\[\int\phi\,\d m_i=\lim_n\int\phi\partial_i u_n\,\d x=-\lim_n\int\partial_i\phi u_n\,\d x=-\int\partial_i \phi f\,\d x,\]
for every $\phi\in C_0^\infty(\R^d)$ and every $1\le i\le d$. Hence $m_i=\D_i f$.
Furthermore, since $\nabla u_n=0$ a.e. outside $\Omega$, by H\"older's inequality, for every $\phi\in C_0^\infty(\R^d;\R^d)$,
\begin{equation}\label{tr1}\begin{split}&\frac{1}{p_n}\lrabs{\int u_n\div \phi\,\d x}^{p_n}=\frac{1}{p_n}\lrabs{\int\lrbr{\nabla u_n}{\phi}\,\d x}^{p_n}\\
\le&\Phi_n(u_n)\times\left\{\begin{aligned}&\norm{\phi}_\infty^{p_n}\left(\vol(\supp \phi)\right)^{p_n/q_n},\;&&\text{if}\,\;p_0=1,\\
&\left(\int_{\R^d}\abs{\phi}^{q_n}\,\d x\right)^{p_n/q_n},\;&&\text{if}\,\;p_0>1.\end{aligned}
\right.\end{split}\end{equation}
We pass to the limit $n\to+\infty$ in \eqref{tr1}. Since $u_n\to f$ in $L^1(\R^d)$, the left-hand side converges to
\[\frac{1}{p_0}\lrabs{\int f\div\phi\,\d x}^{p_0}.\]
If $p_0=1$, then $p_n/q_n=p_n-1\to 0$. If $p_0>1$, then $\sup_n q_n<+\infty$ and
\[\abs{\phi}^{q_n}\le\max(1,\norm{\phi}_\infty)^{\sup_n q_n}1_{\supp\phi},\]
so that Lebesgue's dominated convergence theorem applies. We obtain
\begin{equation}\label{tr2}\begin{split}&\frac{1}{p_0}\lrabs{\int f\div \phi\,\d x}^{p_0}\\
\le& \liminf_n \Phi_n(u_n)\times\left\{ \begin{aligned}&\norm{\phi}_\infty,\;&&\text{if}\;p_0=1,\\
&\left(\int_{\R^d}\abs{\phi}^{q_0}\,\d x\right)^{p_0/q_0},\;&&\text{if}\,\;p_0>1.\end{aligned}\right.\end{split}\end{equation}
Taking the supremum over all $\phi$ with $\norm{\phi}_\infty\le 1$ (if $p_0=1$) or
with $\norm{\phi}_{L^{q_0}(\R^d)}\le 1$ (if $p_0>1$) yields, by the Riesz representation theorem in the latter case, that
\begin{equation}\label{tr3}\left.\begin{aligned}&\norm{\D f}(\R^d)&&\;\text{if}\;p_0=1,\\
&\frac{1}{p_0}\int_{\R^d}\abs{\D f}^{p_0}\,\d x&&\;\text{if}\;p_0>1\end{aligned}\right\}\le\liminf_n \Phi_n(u_n).\end{equation}
Since $u_n\rightharpoonup u$ weakly in $L^2(\Omega)$ and $u_n\to f$ in $L^1(\R^d)$, we have for all $\phi\in C_0(\R^d)$,
\[\int u\phi\,\d x=\lim_n\int u_n\phi\,\d x=\int f\phi\,\d x,\]
hence $u=f$ $\d x$-a.e. on $\R^d$ and $\D u=\D f$ in the sense of distributions.

Let $p_0=1$. By \eqref{tr3}, $\norm{\D u}(\Omega)\le\norm{\D u}(\R^d)\le\liminf_n\Phi_n(u_n)<+\infty$. Since $u\in L^2(\Omega)\subset L^1(\Omega)$,
we get that $u\in BV(\Omega)\cap L^2(\Omega)$ and
\[\Phi(u)=\norm{\D u}(\R^d)\le\liminf_n\Phi_n(u_n).\]
We point out that the trace $\gamma_1(u)$ need not vanish. The boundary values are accounted for by the boundary term in $\Phi_1$, see Remark \ref{Rdrem}.

Let $p_0>1$. By \eqref{tr3}, $\D u\in L^{p_0}(\R^d;\R^d)$. Since $u\in L^1(\R^d)$ has compact support in $\ol{\Omega}\subset B$, the Sobolev--Poincar\'e inequality
on $B$ yields that $u\in W^{1,p_0}(\R^d)$, see e.g. \cite{Maz}. As $u=0$ a.e. outside $\Omega$ and $\partial\Omega$ is Lipschitz (so that $\Omega$ has the segment property),
we get that $u\in W^{1,p_0}_0(\Omega)$, see e.g. \cite{Ada}. Hence
\[\Phi(u)=\frac{1}{p_0}\int_\Omega\abs{\nabla u}^{p_0}\,\d x=\frac{1}{p_0}\int_{\R^d}\abs{\D u}^{p_0}\,\d x\le\liminf_n\Phi_n(u_n).\]
Since we can repeat the steps for any subsequence of $\{u_n\}$, we have proved (M1).

Let us prove (M2) from Definition \ref{moscodefi}. Let $u\in L^2(\Omega)$. If $\Phi(u)=+\infty$, we may take $u_n:=u$ for all $n\in\N$. Suppose that $\Phi(u)<+\infty$.

Let $p_0=1$, so that $u\in BV(\Omega)\cap L^2(\Omega)$. Then by Lemma \ref{lilem} (with $p=2$) there is a sequence $\{u_m\}\subset C^\infty_0(\Omega)$
with
\[u_m\to u\;\text{in}\;L^2(\Omega)\;\text{and}\;\norm{\D u_m}(\R^d)\to\norm{\D u}(\R^d)\;\;\text{as}\;m\to+\infty.\]
For each fixed $m\in\N$, we have that $\frac{1}{p_n}\abs{\nabla u_m}^{p_n}\le\max(1,\norm{\nabla u_m}_\infty)^{\sup_n p_n}1_\Omega$.
Hence by Lebesgue's dominated convergence theorem and \eqref{HBV},
\[\Phi_n(u_m)=\frac{1}{p_n}\int_\Omega\abs{\nabla u_m}^{p_n}\,\d x\to\int_\Omega\abs{\nabla u_m}\,\d x=\norm{\D u_m}(\R^d)\;\;\text{as}\;n\to+\infty.\]
An application of Lemma \ref{wunderlemma}, inequality \eqref{att2}, to $a_{n,m}:=\Phi_n(u_m)$ shows that there exists a sequence $\{m_n\}$, $m_n\uparrow+\infty$, such that
\[\limsup_n \Phi_n(u_{m_n})\le\lim_m\norm{\D u_m}(\R^d)=\norm{\D u}(\R^d)=\Phi(u).\]
Since $m_n\to+\infty$, we have that $u_{m_n}\to u$ in $L^2(\Omega)$, which proves (M2). (Together with (M1), we even get that $\lim_n\Phi_n(u_{m_n})=\Phi(u)$.)

Let $p_0>1$, so that $u\in W^{1,p_0}_0(\Omega)\cap L^2(\Omega)$. There is a sequence $\{u_m\}\subset C_0^\infty(\Omega)$ such that $u_m\to u$ both in $W^{1,p_0}(\Omega)$ and in $L^2(\Omega)$.
If $p_0\ge 2$, this follows from the definition of $W^{1,p_0}_0(\Omega)$, since $\Omega$ is bounded.
If $1<p_0<2$, we first approximate $u$ by its truncations $\max(-k,\min(k,u))\in W^{1,p_0}_0(\Omega)\cap L^\infty(\Omega)$, $k\in\N$, which converge to $u$ in both norms.
A bounded function $v\in W^{1,p_0}_0(\Omega)$ can in turn be approximated in $W^{1,p_0}(\Omega)$ by functions in $C_0^\infty(\Omega)$ which are uniformly bounded by $\norm{v}_\infty+1$
and converge a.e., hence also in $L^2(\Omega)$ by dominated convergence.
As above, for each fixed $m$, by dominated convergence,
\begin{align*}\Phi_n(u_m)&\to\frac{1}{p_0}\int_\Omega\abs{\nabla u_m}^{p_0}\,\d x\;\;\text{as}\;n\to+\infty,\quad\text{and}\\
\frac{1}{p_0}\int_\Omega\abs{\nabla u_m}^{p_0}\,\d x&\to\Phi(u)\;\;\text{as}\;m\to+\infty,\end{align*}
and Lemma \ref{wunderlemma} yields the claim as in the case $p_0=1$.
\end{proof}

Let us briefly treat the case of Neumann boundary conditions. The idea of proof is essentially the
same with slightly different methods. Therefore, let
$p_0\in [1,+\infty)$, $p_n\in (1,+\infty)$, $n\in\N$ such that $\lim_n p_n=p_0$.
Write $q_n:=p_n/(p_n-1)$, $n\in\N$, $q_0:=p_0/(p_0-1)$ (with $1/0:=+\infty$).
In order to apply Theorem \ref{Moscothm},
we would like to prove that
\begin{equation}\label{toprove2}
\Psi_n:=\Psi_{p_n}\xrightarrow[n\to+\infty]{M}\Psi_{p_0}=:\Psi.\end{equation}

We need the following approximation result.
\begin{lemma}\label{aalem}
Let $u\in BV(\Omega)\cap L^2(\Omega)$. Then there exists a sequence $\{u_k\}\subset C^\infty(\ol{\Omega})$ such that
\[\norm{u_k-u}_{L^2(\Omega)}\to 0\quad\text{and}\quad\norm{\D u_k}(\Omega)\to\norm{\D u}(\Omega).\]
\end{lemma}
\begin{proof}
With $C^\infty(\ol{\Omega})$ replaced by $C^\infty(\Omega)$ and $L^2$ replaced by $L^1$, this is classical, see e.g. \cite[Theorem 10.1.2]{ABM}.
For approximations which are smooth up to the boundary, we argue as follows.
Since $\Omega$ is bounded with Lipschitz boundary, there is an extension $Eu\in BV(\R^d)$ with compact support such that $Eu=u$ a.e. on $\Omega$ and
$\norm{\D(Eu)}(\partial\Omega)=0$, see \cite[Proposition 3.21]{AFP}. As $Eu$ is constructed by reflection across $\partial\Omega$ in local bi-Lipschitz charts
and a partition of unity, we also have that $Eu\in L^2(\R^d)$.
Let $\rho_\eps$ be a standard mollifier and set $u_\eps:=(Eu*\rho_\eps)\rs{\ol{\Omega}}\in C^\infty(\ol{\Omega})$.
By Young's convolution inequality, $u_\eps\to u$ in $L^2(\Omega)$ as $\eps\to 0$. By lower semi-continuity of the total variation,
$\liminf_{\eps\to 0}\norm{\D u_\eps}(\Omega)\ge\norm{\D u}(\Omega)$. On the other hand,
\[\norm{\D u_\eps}(\Omega)\le\norm{\D(Eu)}(\Omega_\eps),\quad\text{where}\quad\Omega_\eps:=\{x\in\R^d\;\vert\;\dist(x,\Omega)<\eps\},\]
and $\norm{\D(Eu)}(\Omega_\eps)\to\norm{\D(Eu)}(\ol{\Omega})=\norm{\D u}(\Omega)$ as $\eps\to 0$, since $\norm{\D(Eu)}(\partial\Omega)=0$.
\end{proof}

\begin{proof}[Proof of \eqref{toprove2}]
Let us
prove (M1) from Definition \ref{moscodefi}. Let $u_n,u\in L^2(\Omega)$ such that $u_n\rightharpoonup u$ weakly in $L^2(\Omega)$, and suppose that
\[\liminf_n \Psi_{n}(u_n)<+\infty.\]
Extract
a subsequence (also denoted by $\{u_n\}$) such that
\[\lim_n\Psi_n(u_n)=\liminf_n \Psi_n(u_n)\quad\text{and}\quad C:=\sup_n\Psi_n(u_n)<+\infty.\]
Hence $u_n\in W^{1,p_n}(\Omega)$. As in \eqref{Dieq},
\[\norm{u_n}_{L^1(\Omega)}\le\abs{\Omega}^{1/2}\norm{u_n}_{L^2(\Omega)}\quad\text{and}\quad\int_\Omega\abs{\nabla u_n}\,\d x\le\left(p_n\Psi_n(u_n)\right)^{1/p_n}\abs{\Omega}^{1/q_n},\]
so that $\{u_n\}$ is bounded in $W^{1,1}(\Omega)\subset BV(\Omega)$. By \cite[Theorem 5.2.4]{EG}, a subsequence of $\{u_n\}$ converges in $L^1(\Omega)$ to some $f\in BV(\Omega)$,
and, as above, $f=u$ a.e.

The rest of the proof of (M1) works exactly as above except that the test-functions in
$C_0^\infty(\R^d;\R^d)$ are replaced by test-functions in $C_0^\infty(\Omega;\R^d)$, all integrals are taken over $\Omega$, and in \eqref{tr2} we take the supremum
over all $\phi$ such that $\norm{\phi}_{L^{q_0}(\Omega)}\le 1$ instead. If $p_0=1$, we get that $\Psi(u)=\norm{\D u}(\Omega)\le\liminf_n\Psi_n(u_n)$.
If $p_0>1$, we get that $\D u\in L^{p_0}(\Omega;\R^d)$, hence $u\in W^{1,p_0}(\Omega)$ by the Sobolev--Poincar\'e inequality on the bounded Lipschitz domain $\Omega$, see e.g. \cite{Maz},
and $\Psi(u)\le\liminf_n\Psi_n(u_n)$.

The proof of (M2) from Definition \ref{moscodefi} can be completed by
the arguments of the above proof combined with Lemma \ref{aalem} if $p_0=1$, respectively, if $p_0>1$, with the
approximation of $W^{1,p_0}(\Omega)$-functions by $C^\infty(\ol{\Omega})$-functions, see e.g.
\cite[\S 4.2, Theorem 3]{EG}, combined with truncation as above in order to get convergence in $L^2(\Omega)$ as well.
In both cases, the approximating functions have bounded gradients on $\Omega$, so that the dominated convergence argument above applies.
\end{proof}

\section*{Acknowledgements}

The author would like to thank Diego Pallara for his help on approximation
of bounded variation functions during the conference on evolution equations in Schmitten 2010.

The author would like to thank Samuel Littig and Friedemann Schuricht for a
pleasant stay in Dresden that has given the opportunity for several fruitful discussions leading to the
present version of the paper.

The research was funded by the SFB 701 ``Spectral Structures and Topological Methods in Mathematics'', Bielefeld, and by the Forschergruppe 718 ``Analysis and Stochastics in Complex Physical Systems'', Berlin--Leipzig.

\appendix
\section{A diagonal lemma}

\begin{lemma}\label{wunderlemma}
Let $\{a_{n,m}\}_{n,m\in\N}\subset\ol{\R}$ be a doubly indexed sequence of extended real numbers.
Then there exists a map $n\mapsto m(n)$ with $m(n)\uparrow +\infty$ as $n\to+\infty$ such that
\begin{equation}\label{att1}
\liminf_{n\to+\infty} a_{n,m(n)}\ge\liminf_{m\to+\infty}\left[\liminf_{n\to+\infty} a_{n,m}\right],
\end{equation}
or, equivalently
\begin{equation}\label{att2}
\limsup_{n\to+\infty} a_{n,m(n)}\le\limsup_{m\to+\infty}\left[\limsup_{n\to+\infty} a_{n,m}\right].
\end{equation}
\end{lemma}
\begin{proof}
See \cite[Appendix]{AWets} or \cite[Lemma 1.15 et sqq.]{A}.
\end{proof}

% Bibliography generated by BibTeX (style abbrv) from p-goes-to-1.bib
%\bibliographystyle{abbrv}
%\bibliography{p-goes-to-1}

\begin{thebibliography}{10}

\bibitem{Ada}
R.~A. Adams.
\newblock {\em Sobolev Spaces}, volume~65 of {\em Pure Appl. Math.}
\newblock Academic Press, New York, San Francisco, London, 1975.

\bibitem{AFP}
L.~Ambrosio, N.~Fusco, and D.~Pallara.
\newblock {\em Functions of Bounded Variation and Free Discontinuity Problems}.
\newblock Oxford Math. Monogr. Clarendon Press, Oxford University Press,
  Oxford, 2000.

\bibitem{ABCM}
F.~Andreu, C.~Ballester, V.~Caselles, and J.~M. Maz{\'o}n.
\newblock The {D}irichlet problem for the total variation flow.
\newblock {\em J. Funct. Anal.}, 180(2):347--403, 2001.

\bibitem{ACM}
F.~Andreu-Vaillo, V.~Caselles, and J.~M. Maz{\'o}n.
\newblock {\em Parabolic Quasilinear Equations Minimizing Linear Growth
  Functionals}, volume 223 of {\em Progr. Math.}
\newblock Birkh{\"a}user, Basel, 2004.

\bibitem{A2}
H.~Attouch.
\newblock Convergence de fonctions convexes, des sous-diff{\'e}rentiels et
  semigroupes associ{\'e}s.
\newblock {\em C. R. Acad. Sci. Paris S{\'e}r. A}, 284(10):539--542, 1977.

\bibitem{A3}
H.~Attouch.
\newblock Convergence de fonctionnelles convexes.
\newblock In {\em Journ{\'e}es d'Analyse Non Lin{\'e}aire (Besan{\c{c}}on,
  1977)}, volume 665 of {\em Lecture Notes in Math.}, pages 1--40. Springer,
  Berlin, Heidelberg, New York, 1978.

\bibitem{A}
H.~Attouch.
\newblock {\em Variational Convergence for Functions and Operators}.
\newblock Pitman, Boston, London, Melbourne, 1984.

\bibitem{A4}
H.~Attouch.
\newblock Viscosity solutions of minimization problems.
\newblock {\em SIAM J. Optim.}, 6(3):769--806, 1996.

\bibitem{AttBeer}
H.~Attouch and G.~Beer.
\newblock On the convergence of subdifferentials of convex functions.
\newblock {\em Arch. Math. (Basel)}, 60(4):389--400, 1993.

\bibitem{ABM}
H.~Attouch, G.~Buttazzo, and G.~Michaille.
\newblock {\em Variational Analysis in {S}obolev and {BV} Spaces: Applications
  to {PDE}s and Optimization}, volume~6 of {\em MPS-SIAM Ser. Optim.}
\newblock SIAM, MPS, Philadelphia, 2006.

\bibitem{AttCom}
H.~Attouch and R.~Cominetti.
\newblock {$L^p$} approximation of variational problems in {$L^1$} and
  {$L^\infty$}.
\newblock {\em Nonlinear Anal.}, 36(3):373--399, 1999.

\bibitem{AWets}
H.~Attouch and R.~J.-B. Wets.
\newblock A convergence theory for saddle functions.
\newblock {\em Trans. Amer. Math. Soc.}, 280(1):1--41, 1983.

\bibitem{AuKo}
G.~Aubert and P.~Kornprobst.
\newblock {\em Mathematical Problems in Image Processing: Partial Differential
  Equations and the Calculus of Variations}, volume 147 of {\em Appl. Math.
  Sci.}
\newblock Springer, New York, 2 edition, 2006.

\bibitem{Barb2}
V.~Barbu.
\newblock {\em Analysis and Control of Nonlinear Infinite Dimensional Systems},
  volume 190 of {\em Math. Sci. Eng.}
\newblock Academic Press, Boston, 1993.

\bibitem{BDPR}
V.~Barbu, G.~Da~Prato, and M.~R{\"o}ckner.
\newblock Stochastic nonlinear diffusion equations with singular diffusivity.
\newblock {\em SIAM J. Math. Anal.}, 41(3):1106--1120, 2009.

\bibitem{BaRoe}
V.~Barbu and M.~R{\"o}ckner.
\newblock Stochastic variational inequalities and applications to the total
  variation flow perturbed by linear multiplicative noise.
\newblock {\em Arch. Ration. Mech. Anal.}, 209(3):797--834, 2013.

\bibitem{B}
H.~Bauer.
\newblock {\em Wahrscheinlichkeitstheorie und {G}rundz{\"u}ge der
  {M}a{\ss}theorie}.
\newblock de Gruyter, Berlin, New York, 2 edition, 1974.

\bibitem{BeckSchmidt}
L.~Beck and T.~Schmidt.
\newblock Convex duality and uniqueness for {BV}-minimizers.
\newblock {\em J. Funct. Anal.}, 268(10):3061--3107, 2015.

\bibitem{Boegelein2013}
V.~B{\"o}gelein, F.~Duzaar, and P.~Marcellini.
\newblock Parabolic systems with $p,q$-growth: a variational approach.
\newblock {\em Arch. Ration. Mech. Anal.}, 210(1):219--267, 2013.

\bibitem{Boegelein2014}
V.~B{\"o}gelein, F.~Duzaar, and P.~Marcellini.
\newblock Existence of evolutionary variational solutions via the calculus of
  variations.
\newblock {\em J. Differential Equations}, 256(12):3912--3942, 2014.

\bibitem{Boegelein2019}
V.~B{\"o}gelein, F.~Duzaar, L.~Sch{\"a}tzler, and C.~Scheven.
\newblock Existence for evolutionary problems with linear growth by stability
  methods.
\newblock {\em J. Differential Equations}, 266(11):7709--7748, 2019.

\bibitem{Boegelein2016}
V.~B{\"o}gelein, F.~Duzaar, and C.~Scheven.
\newblock The total variation flow with time dependent boundary values.
\newblock {\em Calc. Var. Partial Differential Equations}, 55(4):Art. 108,
  2016.

\bibitem{Brez}
H.~Br{\'e}zis.
\newblock {\em Op{\'e}rateurs Maximaux Monotones et Semi-groupes de
  Contractions dans les Espaces de {H}ilbert}, volume~5 of {\em North-Holland
  Math. Stud.}
\newblock North-Holland, Amsterdam, London, 1973.

\bibitem{Butt}
G.~Buttazzo.
\newblock {\em Semicontinuity, Relaxation and Integral Representation in the
  Calculus of Variations}, volume 207 of {\em Pitman Res. Notes Math. Ser.}
\newblock Longman Scientific \& Technical, Harlow, 1989.

\bibitem{CioToe1}
I.~Ciotir and J.~M. T{\"o}lle.
\newblock Convergence of invariant measures for singular stochastic diffusion
  equations.
\newblock {\em Stochastic Process. Appl.}, 122(4):1998--2017, 2012.

\bibitem{CioToeCorr}
I.~Ciotir and J.~M. T{\"o}lle.
\newblock Corrigendum to ``{C}onvergence of invariant measures for singular
  stochastic diffusion equations'' [{S}tochastic {P}rocess. {A}ppl. 122 (2012)
  1998--2017].
\newblock {\em Stochastic Process. Appl.}, 123(3):1178--1181, 2013.

\bibitem{CioToe2}
I.~Ciotir and J.~M. T{\"o}lle.
\newblock Nonlinear stochastic partial differential equations with singular
  diffusivity and gradient {S}tratonovich noise.
\newblock {\em J. Funct. Anal.}, 271(7):1764--1792, 2016.

\bibitem{CT}
C.~Combari and L.~Thibault.
\newblock On the graph convergence of subdifferentials of convex functions.
\newblock {\em Proc. Amer. Math. Soc.}, 126(8):2231--2240, 1998.

\bibitem{Diaz}
J.~I. D{\'i}az.
\newblock {\em Nonlinear Partial Differential Equations and Free Boundaries.
  {V}ol. {I}: Elliptic Equations}, volume 106 of {\em Res. Notes Math.}
\newblock Pitman, Boston, 1985.

\bibitem{DiBe}
E.~DiBenedetto.
\newblock {\em Degenerate Parabolic Equations}.
\newblock Universitext. Springer, New York, 1993.

\bibitem{Evan}
L.~C. Evans.
\newblock {\em Partial Differential Equations}, volume~19 of {\em Grad. Stud.
  Math.}
\newblock American Mathematical Society, Providence, 1998.

\bibitem{EG}
L.~C. Evans and R.~F. Gariepy.
\newblock {\em Measure Theory and Fine Properties of Functions}.
\newblock Stud. Adv. Math. CRC Press, Boca Raton, 1992.

\bibitem{Fri}
V.~Fridman.
\newblock {\em Das {E}igenwertproblem zum $p$-{L}aplace {O}perator f{\"u}r $p$
  gegen $1$}.
\newblock PhD thesis, Universit{\"a}t zu K{\"o}ln, 2003.
\newblock URN: urn:nbn:de:hbz:38-10570.

\bibitem{GessRoe}
B.~Gess and M.~R{\"o}ckner.
\newblock Stochastic variational inequalities and regularity for degenerate
  stochastic partial differential equations.
\newblock {\em Trans. Amer. Math. Soc.}, 369(5):3017--3045, 2017.

\bibitem{Gess2014}
B.~Gess and J.~M. T{\"o}lle.
\newblock Multi-valued, singular stochastic evolution inclusions.
\newblock {\em J. Math. Pures Appl. (9)}, 101(6):789--827, 2014.

\bibitem{Gess2016b}
B.~Gess and J.~M. T{\"o}lle.
\newblock Ergodicity and local limits for stochastic local and nonlocal
  $p$-{L}aplace equations.
\newblock {\em SIAM J. Math. Anal.}, 48(6):4094--4125, 2016.

\bibitem{GessToe}
B.~Gess and J.~M. T{\"o}lle.
\newblock Stability of solutions to stochastic partial differential equations.
\newblock {\em J. Differential Equations}, 260(6):4973--5025, 2016.

\bibitem{GianazzaKlaus}
U.~Gianazza and C.~Klaus.
\newblock $p$-parabolic approximation of total variation flow solutions.
\newblock {\em Indiana Univ. Math. J.}, 68(5):1519--1550, 2019.

\bibitem{GKY}
Y.~Giga, Y.~Kashima, and N.~Yamazaki.
\newblock Local solvability of a constrained gradient system of total
  variation.
\newblock {\em Abstr. Appl. Anal.}, 2004(8):651--682, 2004.

\bibitem{KaSchu}
B.~Kawohl and F.~Schuricht.
\newblock Dirichlet problems for the $1$-{L}aplace operator, including the
  eigenvalue problem.
\newblock {\em Commun. Contemp. Math.}, 9(4):515--543, 2007.

\bibitem{Kinnunen2022}
J.~Kinnunen and C.~Scheven.
\newblock On the definition of solution to the total variation flow.
\newblock {\em Calc. Var. Partial Differential Equations}, 61(1):Paper No. 40,
  2022.

\bibitem{Lad}
O.~A. Lady{\v{z}}enskaja.
\newblock New equations for the description of the motions of viscous
  incompressible fluids, and global solvability for their boundary value
  problems ({R}ussian).
\newblock {\em Trudy Mat. Inst. Steklov.}, 102:85--104, 1967.

\bibitem{Latorre2022}
M.~Latorre and S.~Segura~de Le{\'o}n.
\newblock Existence and uniqueness for the inhomogeneous 1-{L}aplace evolution
  equation revisited.
\newblock {\em Rev. R. Acad. Cienc. Exactas F{\'\i}s. Nat. Ser. A Mat. RACSAM},
  116(4):Paper No. 185, 2022.

\bibitem{Latorre2026}
M.~Latorre and S.~Segura~de Le{\'o}n.
\newblock The inhomogeneous total variation flow with {$L^1$}-data.
\newblock {\em Adv. Calc. Var.}, 19(4):707--731, 2026.

\bibitem{Leo}
G.~Leoni.
\newblock {\em A First Course in {S}obolev Spaces}, volume 105 of {\em Grad.
  Stud. Math.}
\newblock American Mathematical Society, Providence, 2009.

\bibitem{Li2025}
D.~Li, Y.~Li, and C.~Zhang.
\newblock Renormalized and entropy solutions for the fractional 1-{L}aplacian
  parabolic equation.
\newblock {\em J. Evol. Equ.}, 25(3):Paper No. 59, 2025.

\bibitem{Li2024}
D.~Li and C.~Zhang.
\newblock Existence and uniqueness of weak solutions to a parabolic nonlocal
  1-{L}aplacian equation.
\newblock arXiv:2406.19020, 2024.

\bibitem{Li}
S.~Littig.
\newblock {\"U}ber das {E}igenwertproblem des $1$-{L}aplace-{O}perators.
\newblock Diploma thesis, Technische Universit{\"a}t Dresden, 2010.

\bibitem{LiSchu}
S.~Littig and F.~Schuricht.
\newblock Convergence of the eigenvalues of the $p$-{L}aplace operator as $p$
  goes to $1$.
\newblock {\em Calc. Var. Partial Differential Equations}, 49(1-2):707--727,
  2014.

\bibitem{Liu1}
W.~Liu.
\newblock On the stochastic $p$-{L}aplace equation.
\newblock {\em J. Math. Anal. Appl.}, 360(2):737--751, 2009.

\bibitem{Maz}
V.~G. Maz'ja.
\newblock {\em Sobolev Spaces}.
\newblock Springer Ser. Soviet Math. Springer, Berlin, Heidelberg, New York,
  Tokyo, 1985.

\bibitem{Mazon2023}
J.~M. Maz{\'o}n, M.~Solera-Diana, and J.~J. Toledo-Melero.
\newblock {\em Variational and Diffusion Problems in Random Walk Spaces},
  volume 103 of {\em Progr. Nonlinear Differential Equations Appl.}
\newblock Birkh{\"a}user, Cham, 2023.

\bibitem{MRSdLT}
A.~Mercaldo, J.~D. Rossi, S.~Segura~de Le{\'o}n, and C.~Trombetti.
\newblock Anisotropic $p,q$-{L}aplacian equations when $p$ goes to $1$.
\newblock {\em Nonlinear Anal.}, 73(11):3546--3560, 2010.

\bibitem{MSdLT}
A.~Mercaldo, S.~Segura~de Le{\'o}n, and C.~Trombetti.
\newblock On the behaviour of the solutions to $p$-{L}aplacian equations as $p$
  goes to $1$.
\newblock {\em Publ. Mat.}, 52(2):377--411, 2008.

\bibitem{MSdLT2}
A.~Mercaldo, S.~Segura~de Le{\'o}n, and C.~Trombetti.
\newblock On the solutions to $1$-{L}aplacian equation with {$L^1$} data.
\newblock {\em J. Funct. Anal.}, 256(8):2387--2416, 2009.

\bibitem{Meyer2026}
D.~Meyer.
\newblock Total $\mathbb{A}$-variation flows.
\newblock {\em Ann. Mat. Pura Appl. (4)}, 205(5):2819--2856, 2026.

\bibitem{Par}
E.~Parini.
\newblock {\em Asymptotic behaviour of higher eigenfunctions of the
  $p$-{L}aplacian as $p$ goes to $1$}.
\newblock PhD thesis, Universit{\"a}t zu K{\"o}ln, 2009.
\newblock Verlag Dr. Hut, M{\"u}nchen (2009), URN: urn:nbn:de:hbz:38-29479.

\bibitem{Par2}
E.~Parini.
\newblock The second eigenvalue of the $p$-{L}aplacian as $p$ goes to $1$.
\newblock {\em Int. J. Differ. Equ.}, 2010:Art. ID 984671, 23 pp., 2010.

\bibitem{Schaetzler2019}
L.~Sch{\"a}tzler.
\newblock Existence for evolutionary {N}eumann problems with linear growth by
  stability results.
\newblock {\em Ann. Acad. Sci. Fenn. Math.}, 44(2):1055--1092, 2019.

\bibitem{Schu}
F.~Schuricht.
\newblock An alternative derivation of the eigenvalue equation for the
  $1$-{L}aplace operator.
\newblock {\em Arch. Math. (Basel)}, 87(6):572--577, 2006.

\bibitem{SeguradeLeon2015}
S.~Segura~de Le{\'o}n and C.~M. Webler.
\newblock Global existence and uniqueness for the inhomogeneous 1-{L}aplace
  evolution equation.
\newblock {\em NoDEA Nonlinear Differential Equations Appl.}, 22(5):1213--1246,
  2015.

\bibitem{Seib2023}
F.~Seib, W.~Stannat, and J.~M. T{\"o}lle.
\newblock Stability and moment estimates for the stochastic singular
  {$\Phi$}-{L}aplace equation.
\newblock {\em J. Differential Equations}, 377:663--693, 2023.

\bibitem{Toe2}
J.~M. T{\"o}lle.
\newblock {\em Variational convergence of nonlinear partial differential
  operators on varying {B}anach spaces}.
\newblock PhD thesis, Universit{\"a}t Bielefeld, Bielefeld, 2010.
\newblock URN: urn:nbn:de:hbz:361-16758.

\bibitem{Toe3}
J.~M. T{\"o}lle.
\newblock Stochastic evolution equations with singular drift and gradient noise
  via curvature and commutation conditions.
\newblock {\em Stochastic Process. Appl.}, 130(5):3220--3248, 2020.

\bibitem{Vaz2}
J.~L. V{\'a}zquez.
\newblock {\em Smoothing and Decay Estimates for Nonlinear Diffusion Equations:
  Equations of Porous Medium Type}, volume~33 of {\em Oxford Lecture Ser. Math.
  Appl.}
\newblock Oxford University Press, Oxford, 2006.

\end{thebibliography}

\end{document}